\documentclass[11pt,a4paper]{article}
\usepackage{amssymb}
\usepackage{amsfonts}

\newtheorem{theorem}{Theorem}

\newtheorem{definition}[theorem]{Definition}

\newtheorem{lemma}[theorem]{Lemma}

\newtheorem{proposition}[theorem]{Proposition}

\newenvironment{proof}[1][Proof]{\noindent\textbf{#1.} }{\ \rule{0.5em}{0.5em}}
\begin{document}

\title{The Khavinson-Shapiro conjecture for domains with a boundary
consisting of algebraic hypersurfaces }
\date{}
\author{Hermann Render \\
\noindent School of Mathematics and Statistics, \noindent University College
Dublin, \\
Belfield, Dublin 4, Ireland.}
\maketitle

\begin{abstract}
The Khavinson-Shapiro conjecture states that ellipsoids are the only bounded
domains in euclidean space satisfying the following property (KS): the
solution of the Dirichlet problem for polynomial data is polynomial. In this
paper we show that a domain does not have property (KS) provided the
boundary contains at least three differrent irreducible algebraic
hypersurfaces for which two of them have a common point.
\end{abstract}

\section{Introduction}

\footnotetext[1]{%
2010 \textit{Mathematics Subject Classification } 31B05; 35J05. Keywords and
phrases: Dirichlet problem, harmonic extension, Khavinson-Shapiro conjecture.%
}

A complex-valued function $h$ defined on an open set $\Omega $ in $\mathbb{R}%
^{d}$ is called \emph{harmonic} if $h$ is differentiable of order $2$ and
satisfies the equation $\Delta h\left( x\right) =0$ for all $x\in \Omega $
where 
\[
\Delta =\frac{\partial ^{2}}{\partial x_{1}^{2}}+\cdots +\frac{\partial ^{2}%
}{\partial x_{d}^{2}} 
\]%
is the Laplacian. The \emph{Khavinson-Shapiro conjecture }in \cite{KhSh92}
states that ellipsoids are the only bounded domains $\Omega $ in $\mathbb{R}%
^{d}$ with the following property:

\begin{itemize}
\item[(KS)] For any polynomial $p$ there exists a harmonic polynomial $h$
such that $h\left( \xi\right) =p\left( \xi\right) $ for all $\xi\in
\partial\Omega.$
\end{itemize}

Obviously a domain $\Omega $ has property (KS) if and only if the Dirichlet
problem for polynomial data (restricted to the boundary) has polynomial
solutions; for more details on the Dirichlet problem we refer the reader to 
\cite{ArGa01} and \cite{Gard93}. The Khavinson-Shapiro conjecture has been
confirmed for large classes of domains but it is still unproven in its full
generality, and we refer the interested reader to \cite{HaSh94}, \cite%
{Rend08}, \cite{Lund}, \cite{KhLu10}, \cite{LuRe10}, \cite{KhLu14} and for
further ramifications in \cite{KhSt} originating from the work in \cite%
{PutSty}.

For certain classes of domains $\Omega $ one can prove that the solution of
the Dirichlet problem for the polynomial 
\[
\left\vert x\right\vert ^{2}=x_{1}^{2}+\cdots +x_{d}^{2} 
\]%
develops a singularity in $\mathbb{R}^{d},$ in particular it is not a
polynomial and $\Omega $ does not have property (KS). Instructive is a
beautiful result of P. Ebenfeld in \cite{Eben92} for the "TV-screen" domain 
\[
\Omega =\left\{ \left( x,y\right) \in \mathbb{R}^{2}:x^{4}+y^{4}<1\right\} 
\]%
where the solution of the Dirichlet problem $u\left( x,y\right) $ for the
data function $x^{2}+y^{2}$ can be extended to a real-analytic function on $%
\mathbb{R}^{2}\setminus D$ where $D$ is a countable discrete set. Further
illuminating examples of domains satisfying the Khavinson-Shapiro conjecture
have been provided by E. Lundberg in \cite{Lund} using lightning bolts as a
mayor technique.

In this paper we want to discuss the Khavinson-Shapiro conjecture for
domains which consists of finitely many algebraic hypersurfaces. The
simplest example of a such domain is a polygonal domain in the plane. E.A.
Volkov proved in \cite{Volk99} the following result: \emph{An equilateral
triangle is the unique polygon }$D$\emph{\ in the plane xy on which the
solution of the boundary problem }%
\begin{equation}
\Delta u=1\hbox{ on }D\hbox{ and }u=0\hbox{ on }\partial D  \label{eqvol}
\end{equation}%
\emph{is an algebraic polynomial, and moreover, the degree of this
polynomial is equal to }$3.$ Note that for a solution $u$ of (\ref{eqvol})
the function $v\left( x,y\right) =x^{2}+y^{2}-4u\left( x,y\right) $ is
harmonic and equal to $x^{2}+y^{2}$ on $\partial D.$ It follows that the
only polygonal domain for which the polynomial $x^{2}+y^{2}$ has a
polynomial solution for the Dirichlet problem is an equilateral triangle. We
shall prove in this paper that no polygonal domain has property (KS). Thus
the equilateral triangle is an example of a domain where the polynomial $%
x^{2}+y^{2}$ has a polynomial solution for the Dirichlet problem but
property (KS) is not satisfied (for further examples see \cite{HaSh94} or 
\cite{LuRe10}).

Let us denote the set of all polynomials in $d$ variables with real
coefficients by $\mathbb{R}\left[ x\right] $. Assume that $\psi
_{1},...,\psi _{N}\in \mathbb{R}\left[ x\right] .$ We say that the boundary $%
\partial \Omega $ of a domain $\Omega $ in $\mathbb{R}^{d}$ contains $N$
different irreducible algebraic hypersurfaces defined by $\psi _{1},...,\psi
_{N}$ if (i) each polynomial $\psi _{k},k=1,...N$ is irreducible, (ii) $\psi
_{k}\neq \lambda \psi _{l}$ for all $k\neq l$, $k,l=1,...,N$ and for all
real numbers $\lambda $, and (iii) for each $k=1,...,N$ there exists an open
set $U_{k}$ such that 
\[
\left\{ x\in \mathbb{R}^{d}:\psi _{k}\left( x\right) =0\right\} \cap
U_{k}\subset \partial \Omega 
\]%
and $\psi _{k}$ changes its sign on $U_{k},$ i.e. there exist $x,y\in U_{k}$
such that $\psi _{k}\left( x\right) <0<\psi _{k}\left( y\right) .$ Obviously
a polygonal domain in the plane with $N$ edges satisfies this condition
where the polynomials $\psi _{k},k=1,...,N$ are linear polynomials.

Elementary results in Algebraic Geometry (cf. the proof of Theorem 27 in 
\cite{Rend08}) lead to the following observation:

\begin{theorem}
Assume that $\partial \Omega $ contains $N$ different irreducible algebraic
hypersurfaces defined by the polynomials $\psi _{1},...,\psi _{N}$. If $%
\Omega $ has property (KS) then the Fischer operator\footnote[2]{%
More generaly one can define a Fischer operator by $q\longmapsto P\left(
D\right) \left( \psi q\right) $ where $P\left( D\right) $ is a linear
partial differential operator with constant real coefficients. Fischer's
Theorem in \cite{Fischer} states that the Fischer operator is a bijection
whenever $\psi \left( x\right) $ is a homogeneous polynomial equal to the
polynomial $P\left( x\right) .$} $F_{\psi }:$ $\mathbb{R}\left[ x\right]
\rightarrow \mathbb{R}\left[ x\right] $ defined by 
\[
F_{\psi }\left( q\right) :=\Delta \left( \psi q\right) \hbox{ for }q\in 
\mathbb{R}\left[ x\right] 
\]%
is surjective where $\psi =$ $\psi _{1}\cdot ...\cdot \psi _{N}$.
\end{theorem}

The main result of this paper is the following:

\begin{theorem}
Suppose that the polynomial $\psi$ has at least three non-constant
polynomial factors and two of the factors have a common zero $x_{0}.$ Then $%
F_{\psi}$ is not surjective.
\end{theorem}

As an application we see that no bounded domain $\Omega $ in $\mathbb{R}^{2}$
with polygonal boundary has property (KS): there are at least $3$ edges
defined by linear polynomials $\psi _{1},\psi _{2},\psi _{3}.$ Since $\Omega 
$ is bounded at least two of them have a common zero and the result follows.
In particular, the equilateral triangle does not satisfy property (KS).

Theorem 1 shows that (KS) is related to the following purely algebraic
question: for which polynomials $\psi \in \mathbb{R}\left[ x\right] $ is the
Fischer operator $F_{\psi }:$ $\mathbb{R}\left[ x\right] \rightarrow \mathbb{%
R}\left[ x\right] $ defined by 
\[
F_{\psi }\left( q\right) :=\Delta \left( \psi q\right) \hbox{ for }q\in 
\mathbb{R}\left[ x\right] 
\]%
surjective? In \cite{ChSi01} M. Chamberland and D. Siegel formulated the
following conjecture:

\begin{itemize}
\item[(CS)] The surjectivity of the Fischer operator $F_{\psi}:$ $\mathbb{R}%
\left[ x\right] \rightarrow\mathbb{R}\left[ x\right] $ implies that the
degree of $\psi$ is $\leq2.$
\end{itemize}

In the last section we will summarize a recent result of the author showing
that (CS) implies the Khavinson-Shapiro conjecture for so-called admissible
domains. We refer the interested reader to \cite{LuRe10} for more results on
the conjecture (CS) and related statements.

\section{Surjectivity of Fischer operators.}

The degree of a polynomial $\psi $ is denoted by $\deg \psi .$ Recall that
the multiplicity $N$ of a zero $x_{0}$ of a polynomial $\psi \left( x\right) 
$ is the largest natural number $N$ such that 
\[
\frac{\partial ^{\alpha }}{\partial x^{\alpha }}\psi \left( x_{0}\right) =0%
\hbox{ for all }\alpha =\left( \alpha _{1},...,\alpha _{d}\right) \in 
\mathbb{N}_{0}^{d}\hbox{ with }\left\vert \alpha \right\vert =\alpha
_{1}+...+\alpha _{d}\leq N-1. 
\]

A polynomial $f_{j}$ is called \emph{homogeneous} of degree $j$ if $%
f_{j}(rx)=r^{j}f_{j}(x)$ for all $r>0$ and for all $x\in \mathbb{R}^{d}$. A
polynomial $f$ of degree $m$ can be expanded into a finite sum of
homogeneous polynomials $f_{j}$ of degree $j$ such that 
\[
f=f_{t}+...+f_{m}\hbox{ and }t\leq m 
\]%
where $f_{m}\neq 0$ is called the \emph{principal part} or \emph{leading part%
}$,$ and $f_{t}\neq 0$ is called the \emph{initial part} of the polynomial $%
f.$

The following definition arises in the investigation of stationary sets for
the wave and heat equation, see \cite{AgKr00},\cite{AgQu01}, and the
injectivity of the spherical Radon transform, see \cite{AVZ99}, and the
investigation of level curves of harmonic functions, see \cite[p. 425]{FNS66}%
.

\begin{definition}
A polynomial $f$ is called a harmonic divisor if there exists a non-zero
polynomial $q$ such that $fq$ is harmonic.
\end{definition}

Thus a polynomial is harmonic divisor if and only if is a factor of a
harmonic polynomial, in particular every harmonic polynomial is a harmonic
divisor (by taking $q$ as the constant function).

By Theorem 2.2.3 in \cite{ArGa01} any harmonic divisor takes both positive
and negative values, thus a harmonic divisor does not have any non-negative
factor, a result which is due to M. Brelot and G. Choquet in \cite{B-C}.

The following result is well known:

\begin{proposition}
Let $f$ be a polynomial with homogeneous expansion $f=f_{t}+...+f_{m},$
where $f_{t}\neq0$ and $f_{m}\neq0.$ If $f$ is a harmonic divisor then $%
f_{t} $ and $f_{m}$ are harmonic divisors.
\end{proposition}

The next two results provide some simple necessary conditions of the
surjectivity of the Fischer operator:

\begin{proposition}
\label{propharDiv}Let $\psi$ be a polynomial of degree $\geq3$ with
homogeneous expansion 
\[
\psi=\psi_{t}+...+\psi_{m}, 
\]
where $\psi_{t}\neq0$ and $\psi_{m}\neq0.$ If the constant function is in
the image of the Fischer operator $F_{\psi}$ then $\psi_{m}$ is a harmonic
divisor.
\end{proposition}

\begin{proof}
For the constant function $1$ there exists a polynomial $q$ such that $%
1=\Delta\left( \psi q\right) .$ Let us write $q=q_{T}+....+q_{M}$ with $%
q_{T}\neq0$ and $q_{M}\neq0.$ Then the polynomial $\psi q$ has the
homogeneous expansion 
\[
\psi q=\psi_{t}q_{T}+\left( \psi_{t+1}q_{T}+\psi_{t}q_{T+1}\right)
+\cdots+\psi_{m}q_{M}. 
\]
By assumption $1=\Delta\left( \psi q\right) ,$ i.e. the homogeneous
expansion of the polynomial $\Delta\left( \psi q\right) $ consists just of
the constant function $1.$ On the other hand, $\psi_{m}q_{M}$ has degree at
least $3$, so $\Delta\left( \psi_{m}q_{M}\right) $ has degree $1,$and
therefore $\Delta\left( \psi_{m}q_{M}\right) =0.$
\end{proof}

\begin{proposition}
\label{propositionOrder2}Let $\psi$ be a polynomial. If the constant
function is in the image of the Fischer operator $F_{\psi}$ then any zero $%
x\in\mathbb{R}^{d}$ of $\psi$ has multiplicity $\leq2$.
\end{proposition}

\begin{proof}
Let $q$ be a polynomial such that $1=F_{\psi}\left( q\right) =\Delta\left(
\psi q\right) .$ If the multiplicity of $\psi$ at $x_{0}$ is larger than $2$
then all derivatives of $\psi$ of order $\leq2$ at $x_{0}$ are zero and
therefore 
\[
\frac{\partial^{2}}{\partial x_{j}^{2}}\left( \psi q\right) \left(
x_{0}\right) =0 
\]
for $j=1,...,d.$ Thus $\Delta\left( \psi q\right) \left( x_{0}\right) =0,$ a
contradiction.
\end{proof}

From \cite{LuRe10} we cite the following result which is known as the \emph{%
Fischer decomposition} of a polynomial.

\begin{proposition}
\label{connect} Suppose $\psi$ is a polynomial. Then the operator $F_{\psi
}\left( q\right) :=\Delta\left( \psi q\right) $ is surjective if and only if
every polynomial $f$ can be decomposed as $f=\psi q_{f}+h_{f}$, where $q_{f}$
is a polynomial and $h_{f}$ is harmonic polynomial
\end{proposition}

In the following we always assume that the polynomial $\psi $ in the
definition of the Fischer operator $F_{\psi }$ has at least two factors:

\begin{proposition}
\label{propositionHD}Suppose that the polynomial $\psi$ has at least two
non-constant polynomial factors, say $\psi=\psi^{\left( 1\right)
}\psi^{\left( 2\right) }.$ If $F_{\psi}$ is surjective then $\psi^{\left(
1\right) }$ and $\psi^{\left( 2\right) }$ are harmonic divisors.
\end{proposition}

\begin{proof}
We can write $\psi^{\left( 1\right) }=\psi^{\left( 1\right) }\psi^{\left(
2\right) }q+h$ for some polynomial $q$ and a harmonic polynomial $q.$ Then $%
h\neq0$ since otherwise $\psi^{\left( 1\right) }=\psi^{\left( 1\right)
}\psi^{\left( 2\right) }q$ which is impossible since $\psi^{\left( 2\right)
} $ has degree $\geq1$. It follows that for $g:=1-\psi^{\left( 2\right) }q$ 
\[
\psi^{\left( 1\right) }g=h\neq0. 
\]
Then $g\neq0$ and $\psi^{\left( 1\right) }$ is a harmonic divisor. The case
of $\psi^{\left( 2\right) }$ follows by symmetry of the problem.
\end{proof}

\begin{lemma}
\label{LemTrans}Let $x_{0}\in\mathbb{R}^{d}.$ Then the Fischer operator $%
F_{\psi}:\mathbb{R}\left[ x\right] \rightarrow\mathbb{R}\left[ x\right] $ is
surjective if and only if $F_{\varphi}:\mathbb{R}\left[ x\right] \rightarrow%
\mathbb{R}\left[ x\right] $ is surjective for the polynomial $\varphi\left(
x\right) =\psi\left( x-x_{0}\right) $.
\end{lemma}

\begin{proof}
Let $f\left( x\right) $ be polynomial and define $g\left( x\right) =f\left(
x+x_{0}\right) .$ If $F_{\psi}$ is surjective then there exist a polynomial $%
q$ and a harmonic polynomial $h$ such that $g\left( x\right) =\psi\left(
x\right) q\left( x\right) +h\left( x\right) $ for all $x\in\mathbb{R}^{d}.$
Replace $x$ by $x-x_{0},$ then 
\[
f\left( x\right) =\psi\left( x-x_{0}\right) q\left( x-x_{0}\right) +h\left(
x-x_{0}\right) . 
\]
Clearly $x\longmapsto h\left( x-x_{0}\right) $ is harmonic. It follows that $%
F_{\varphi}$ is surjective. The converse follows similarly.
\end{proof}

If $F_{\psi}$ is surjective and if $\psi$ has a zero of multiplicity $2$ at
some $x_{0}$ then Lemma \ref{LemTrans} shows that we may assume that $%
x_{0}=0.$ In this case the following observation is crucial:

\begin{theorem}
\label{theoremMain}Suppose that $\psi$ has a zero of multiplicity $2$ at $0$
and let 
\[
\psi\left( x\right) =\psi_{2}\left( x\right) +...+\psi_{t}\left( x\right) 
\]
be the expansion into a sum of homogeneous polynomials with $\psi_{2}\neq0$
and $\psi_{t}\neq0.$ If the Fischer operator $F_{\psi}:\mathbb{R}\left[ x%
\right] \rightarrow\mathbb{R}\left[ x\right] $ is surjective then the
Fischer operator $F_{\psi_{2}}:\mathbb{R}\left[ x\right] \rightarrow \mathbb{%
R}\left[ x\right] $ with respect to the polynomial $\psi_{2}$ of degree $2$
is bijective. Moreover the Fischer operator $F_{\psi}:\mathbb{R}\left[ x%
\right] \rightarrow\mathbb{R}\left[ x\right] $ is bijective.
\end{theorem}

\begin{proof}
Let $\mathbb{P}^{\leq N}$ be the space of all polynomials of degree $\leq N.$
We claim by induction over $N$ that the operator $F_{\psi_{2}}:\mathbb{P}%
^{\leq N}\rightarrow$ $\mathbb{P}^{\leq N}$ is a bijection. For $N=0,$ the
space $\mathbb{P}^{\leq0}$ consists of multiples of the constant
polynomials. By surjectivity of $F_{\psi}$ there exists a polynomial $q$
such that $1=\Delta\left( \psi q\right) .$ Let us write $q=q_{0}+....+q_{s}.$
Then 
\[
\psi q=\psi_{2}q_{0}+\hbox{higher order terms.} 
\]
Since $1=\Delta\left( \psi q\right) $ we infer that $1=\Delta\left(
\psi_{2}q_{0}\right) =F_{\psi_{2}}\left( q_{0}\right) .$ Thus $F_{\psi_{2}}:%
\mathbb{P}^{\leq0}\rightarrow\mathbb{P}^{\leq0}$ is surjective and therefore
bijective. Assume that the statement is true for $N-1$ and we want to prove
it for $N.$ By induction hypothesis and the linearity of $F_{\psi_{2}}:%
\mathbb{P}^{\leq N}\rightarrow$ $\mathbb{P}^{\leq N}$ it suffices to show
that for each non-zero homogeneous polynomial $f_{N}$ of degree $N$ there
exists $q\in\mathbb{P}^{\leq N}$ with $f_{N}=F_{\psi_{2}}q.$ Since $F_{\psi}$
is surjective there exists a polynomial $q$ such that $F_{\psi}\left(
q\right) =f_{N}.$ We write $q$ has a sum of homogeneous polynomials $q_{j}$,
say 
\[
q=q_{m}+....+q_{M} 
\]
where $q_{m}\neq0$ and $q_{M}\neq0.$ Then 
\begin{equation}
\psi q=\psi_{2}q_{m}+G_{m+3}+....+G_{t+N}  \label{eqGG}
\end{equation}
where $G_{m+j}$ are homogeneous polynomials of degree $m+j$ for $j=3,...,n$.
If $m<N$ then $\Delta\left( \psi_{2}q_{m}\right) =0$ since $\Delta\left(
\psi q\right) =f_{N}$. By induction hypothesis $F_{\psi_{2}}:\mathbb{P}%
^{\leq N-1}\rightarrow$ $\mathbb{P}^{\leq N-1}$ is bijective, so we conclude
that $q_{m}=0,$ a contradiction. Hence $m\geq N.$ If $m>N$ then all
homogeneous summands of $\psi q$ have degree at least $m+2>N,$ and from $%
\Delta\left( \psi q\right) =f_{N}$ we obtain that $f_{N}=0,$ a
contradiction. It follows that $m=N$ and $f_{N}=\Delta\left(
\psi_{2}q_{N}\right) .$ Thus $F_{\psi _{2}}$ is bijective.

Suppose that $F_{\psi}$ is not injective. Then there exists a polynomial $q$
such that $\Delta\left( \psi q\right) =0,$ so $\Delta\left(
\psi_{2}q_{m}\right) =0$ for some homogeneous polynomial $q_{m}\neq0.$ Since 
$F_{\psi_{2}}$ is injective we infer that $q_{m}=0,$ a contradiction.
\end{proof}

\begin{theorem}
\label{theorem2factor}Suppose that the polynomial $\psi$ of degree $\geq3$
has two non-constant polynomial factors, say $\psi=\psi^{\left( 1\right)
}\psi^{\left( 2\right) }.$ If there exists a zero $x_{0}$ of $\psi^{\left(
1\right) }$ of multiplicity $\geq2$ then $F_{\psi}$ is not surjective.
\end{theorem}

\begin{proof}
Suppose that $F_{\psi }$ is surjective. Proposition \ref{propositionOrder2}
show that $x_{0}$ is a zero of $\psi $ of multiplicity $\leq 2.$ Since $\psi
^{\left( 1\right) }$ has a zero of multiplicity $2$ at $x_{0}$ it follows
that $\psi ^{\left( 2\right) }$ has not a zero at $x_{0}.$ We may assume
that $\psi ^{\left( 2\right) }\left( 0\right) =1.$ By Proposition \ref%
{propositionHD} $\psi ^{\left( 1\right) }$ is a harmonic divisor, and by
Proposition \ref{propharDiv} the initial part $\psi _{2}^{\left( 2\right) }$
of the homogeneous expansion of $\psi _{2}^{\left( 1\right) }$ is a harmonic
divisor. Since $\psi ^{\left( 2\right) }\left( 0\right) =1$ it follows that $%
\psi _{2}^{\left( 1\right) }$ is initial part of the homogeneous expansion
of $\psi $ which was denoted by $\psi _{2}.$ Thus $\psi _{2}$ is a harmonic
divisor which contradicts to the fact $F_{\psi _{2}}$ is bijective.
\end{proof}

Finally we present the proof of the main result of the paper:

\begin{theorem}
Suppose that the polynomial $\psi $ has at least three non-constant
polynomial factors and two of the factors have a common zero $x_{0}.$ Then $%
F_{\psi }$ is not surjective.
\end{theorem}

\begin{proof}
Let $\psi =\varphi _{1}\varphi _{2}\varphi _{3}$ for some non-constant
polynomial $\varphi _{j}$ for $j=1,2,3.$ Suppose that $x_{0}$ is a common
zero of two factors, say $\varphi _{1}\left( x_{0}\right) =0$ and $\varphi
_{2}\left( x_{0}\right) =0.$ Then the product $\varphi _{1}\varphi _{2}$ has
a zero at $x_{0}$ of multiplicity $\geq 2.$ By Theorem \ref{theorem2factor}
the Fischer operator is not surjective.
\end{proof}

\section{On the conjecture (CS)}

The zero-set of a polynomial $f$ will be denoted by $Z\left( f\right)
=\left\{ x\in \mathbb{R}^{d}:f\left( x\right) =0\right\} .$ We say that $Z$
is an \emph{admissible common zero set} if there exists \emph{non-constant
irreducible} polynomials $f,g\in \mathbb{R}\left[ x\right] $ such that (i) $%
f\neq \lambda g$ for all $\lambda \in \mathbb{R}$ and (ii) $Z=Z\left(
f\right) \cap Z\left( g\right) .$ Now we define the class of domains for
which we can characterize the property (KS):

\begin{definition}
An open $\Omega$ in $\mathbb{R}^{d}$ is admissible if for any $x\in
\partial\Omega$, any open neighborhood $V$ of $x$ and for any finite family
of admissible common zero sets $Z_{1},...,Z_{r}$ the set 
\[
\left[ \partial\Omega\cap V\right] \diagdown\bigcup_{j=1}^{r}Z_{j} 
\]
is non-empty.
\end{definition}

For dimension $d=2$ it is easy to see that an open set $\Omega $ is
admissible if and only if each boundary point $x\in \partial \Omega $ is not
isolated in $\partial \Omega $ since an admissible common zero set is
finite. For arbitrary dimension $d$ it is intuitively clear that an
admissible common zero set has dimension $d-2.$ Thus if the boundary $%
\partial \Omega $ has dimension $d-1$ we expect that the domain $\Omega $ is
admissible. However, it seems to be difficult to formulate this as a precise
topological condition.

In \cite{Rend15} we presented the following result:

\begin{theorem}
\label{ThmI1}Let $\Omega$ be an open admissible subset of $\mathbb{R}^{d}$.
Then property (KS) holds for $\Omega$ if and only if there exists a
non-constant polynomial $\psi\in\mathbb{R}\left[ x\right] $ such that (i) $%
\partial\Omega\subset\psi^{-1}\left\{ 0\right\} $ and (ii) the Fischer
operator $F_{\psi}:$ $\mathbb{R}\left[ x\right] \rightarrow\mathbb{R}\left[ x%
\right] $ defined by 
\[
F_{\psi}\left( q\right) :=\Delta\left( \psi q\right) \hbox{ for }q\in\mathbb{%
R}\left[ x\right] 
\]
is surjective.
\end{theorem}

It follows that for admissible domains the conjecture (CS) implies the
Khavinson-Shapiro conjecture.


\begin{thebibliography}{99}
\bibitem{AgKr00} M.L. Agranovsky, Y. Krasnov, \emph{Quadratic Divisors of
Harmonic polynomials in }$\mathbb{R}^{n},$ J. D'Anal. Math. 82 (2000),
379--395.

\bibitem{AgQu01} M.L. Agranovsky, E.T. Quinto, \emph{Geometry of stationary
sets for the wave equation in }$\mathbb{R}^{n}.$\emph{\ The case of finitely
supported initial data,} Duke Math. J. 107 (2001), 57--84.

\bibitem{AVZ99} M.L. Agranovsky, V.V. Volchkov, L.A. Zalcman, \emph{Conical
Uniqueness sets for the spherical Radon Transform,} Bull. London Math. Soc.
31 (1999), 231--236.

\bibitem{Arm} D. H. Armitage, \emph{The Dirichlet problem when the boundary
function is entire,} J. Math. Anal. Appl. 291 (2004), 565--577.

\bibitem{ArGa01} D. H. Armitage, S. J. Gardiner, \emph{Classical Potential
Theory}, Springer, London 2001.

\bibitem{Ax} S. Axler, P. Bourdon, W. Ramey, \emph{Harmonic Function Theory}%
, 2nd Edition, Springer, 2001.

\bibitem{AGV03} S. Axler, P. Gorkin, K. Voss, \emph{The Dirichlet problem on
quadratic surfaces,} Math. Comp. 73 (2003), 637--651.

\bibitem{Ba} J. A. Baker, \emph{The Dirichlet problem for ellipsoids,} Amer.
Math. Monthly, 106 (1999), 829--834.

\bibitem{B-C} M. Brelot, G. Choquet, \emph{Polynomes harmoniques et
polyharmoniques,} Second colloque sur les 
\'{}%
equations aux d%
\'{}%
eriv%
\'{}%
ees partielles, Bruxelles, (1954) 45--66.

\bibitem{ChSi01} M. Chamberland, D. Siegel, \emph{Polynomial solutions to
Dirichlet problems,} Proc. Amer. Math. Soc. 129 (2001), 211--217.

\bibitem{Eben92} P. Ebenfelt, \emph{Singularities encountered by the
analytic continuation of solutions to Dirichlet's problem,} Complex
Variables 20 (1992), 75--91.

\bibitem{FNS66} L. Flatto, D.J. Newman, H.S. Shapiro, \emph{The level curves
of harmonic functions,} Trans. Amer. Math. Soc. 123 (1966), 425--436.

\bibitem{Fischer} E. Fischer, \emph{\"{U}ber die Differentiationsprozesse
der Algebra}, J. f\"{u}r Math. (Crelle Journal) 148 (1917), 1--78.

\bibitem{Gard93} S.J. Gardiner, \emph{The Dirichlet problem with non-compact
boundary,} Math. Z. 213 (1993), 163--170.

\bibitem{HaSh94} L. Hansen, H.S. Shapiro, \emph{Functional Equations and
Harmonic Extensions,} Complex Variables, 24 (1994), 121--129.

\bibitem{KhLu10} D. Khavinson, E. Lundberg, \emph{The search for
singularities of solutions to the Dirichlet problem: recent developments},
Hilbert spaces of analytic functions, 121--132, CRM Proc. Lecture Notes, 51,
Amer. Math. Soc., Providence, RI, 2010.

\bibitem{KhLu14} D. Khavinson, E. Lundberg, \emph{A tale of ellipsoids in
Potential Theory,} Notices Amer. Math. Soc. 61 (2014), 148--156.

\bibitem{KhSh92} D. Khavinson, H. S. Shapiro,\emph{\ Dirichlet's Problem
when the data is an entire function,} Bull. London Math. Soc. 24 (1992),
456--468.

\bibitem{KhSt} D. Khavinson, N. Stylianopoulos, \emph{Recurrence relations
for orthogonal polynomials and algebraicity of solutions of the Dirichlet
problem, }to appear in \textquotedblleft Around the Research of Vladimir
Maz'ya II, Partial Differential Equations\textquotedblright, 219--228,
Springer, 2010.

\bibitem{Lund} E. Lundberg, \emph{Dirichlet's problem and complex lightning
bolts,} Comp. Meth. Function Theory 9 (2009), 111--125.

\bibitem{LuRe10} E. Lundberg, H. Render, \emph{The Khavinson-Shapiro
conjecture and polynomial decompositions,} J. Math. Anal. Appl. (2010),
506--513.

\bibitem{PutSty} M. Putinar, N. Stylianopoulos, \emph{Finite-term relations
for planar orthogonal polynomials}, Complex Anal. Oper. Theory 1 (2007),
447--456.

\bibitem{Rend08} H. Render, \emph{Real Bargmann spaces, Fischer
decompositions and sets of uniqueness for polyharmonic functions,} Duke
Math. J. 142 (2008), 313--352.

\bibitem{Rend15} H. Render, \emph{A characterization of the
Khavinson-Shapiro conjecture via Fischer operators,} submitted.

\bibitem{Shap89} H.S. Shapiro, \emph{An algebraic theorem of E. Fischer and
the Holomorphic Goursat Problem,} Bull. London Math. Soc. 21 (1989),
513--537.

\bibitem{Volk65} E.A. Volkov, \emph{On differential properties of solutions
to boundary value problems for the Laplace equation on polygons,} Trudy Mat.
Inst. Steklov [Proc. Steklov Inst. Math.], 77 (1965), 113--142.

\bibitem{Volk99} E.A. Volkov, \emph{On a property of solutions to the
Poisson equation on polygons,} Mathematical Notes 66 (1999), 139--141
(translated from Matematicheskie Zametki 66 (1999), 178--180).
\end{thebibliography}
\end{document}